\pdfoutput=1
\documentclass[10pt,a4paper,english]{article} 
\usepackage{babel}
\usepackage{amsmath,amssymb,amsthm}
\usepackage{hyperref}

\usepackage{csquotes}
\usepackage{mathptmx}

\theoremstyle{remark}
\newtheorem{remark}{Remark}[section]
\theoremstyle{plain}
\newtheorem{theorem}[remark]{Theorem}
\newtheorem{proposition}[remark]{Proposition}
\newtheorem{lemma}[remark]{Lemma}

\theoremstyle{definition}

\let\ge=\varepsilon

\newcommand{\K}{\mathcal{K}}

\begin{document}
    \title{Perturbation results for some
nonlinear equations involving fractional operators}
    \author{Simone Secchi
    \thanks{Partially supported by PRIN
2009 ``Teoria dei punti critici
e metodi perturbativi per equazioni differenziali nonlineari''.}}
\maketitle

\begin{abstract}
\noindent By using a perturbation technique in critical point theory, we prove
the existence of solutions for two types of nonlinear equations involving fractional differential operators.
\vspace{5mm}

\noindent \emph{Keywords:} Perturbation methods, pseudo-relativistic Hartree equation, fractional laplacian.

\noindent \emph{AMS Subject Classification:} 35Q55, 35A15, 35J20
\end{abstract}

    \section{Introduction}
    
    In this note we propose a few existence results for solutions to nonlinear elliptic equations driven by fractional operators. We will focus on two models: a pseudo-relativistic Hartree equation, and an equation involving the fractional laplacian. We refer to the next sections for more details on these problems. 
    
Our approach relies on a perturbation technique in Critical Point Theory introduced some years ago by Ambrosetti and his collaborators. It is very useful when dealing with perturbation problems with lack of compactness. For the reader's sake, we collect in this Introduction the main ingredients of this method. The interested reader will find the complete theory in the book \cite{AM}.

\subsection{The abstract setting}

Consider a (real) Hilbert space $H$, and a family $\{f_\ge\}$ of functionals $f_\ge \colon H \to \mathbb{R}$ of class $C^2$. We assume that
\begin{equation}
f_\ge = f_0 + \ge G,
\end{equation}
where $G \colon H \to \mathbb{R}$ and $f_0 \colon H \to \mathbb{R}$ satisfies
\begin{itemize}
\item[($h_1$)] $f_0 \in C^2(H)$ has a smooth manifold $Z$ of dimension $d < \infty$, such that $f'_0(z)=0$ for every $z \in Z$.
\item[($h_2$)] The linear operator $f_0''(z)$ is a Fredholm operator of index zero, for every $z \in Z$.
\item[($h_3$)] $\ker f_0''(z) = T_z Z$ for every $z \in Z$. Here $T_z Z$ stands for the tangent space at $z$ to the manifold $Z$.
\end{itemize}
If we look for critical points of $f_\ge$ as $\ge \to 0$, i.e. for points $u \in H$ with
\begin{equation}
f_\ge'(u)=0,
\end{equation}
we can \emph{deform} the manifold $Z$ to a new manifold $Z_\ge$ in such a way that $Z_\ge$ is a \emph{natural constraint} for $f_\ge$. Let us briefly recall the construction.

By the Implicit Function Theorem, we can construct a function $w = w(z,\ge)$ with values in $Z$ and such that 
\begin{enumerate}
\item $w(z,0)=0$ for every $z \in Z$;
\item $f_\ge'(z+w(z,\ge))\in T_z Z$ for every $z \in Z$;
\item $w(z,\ge) \in \left( T_z Z \right)\sp\perp$ for every $z \in Z$.
\end{enumerate}
It follows without effort that $w=O(\ge)$ as $\ge \to 0$.
Then we introduce the perturbed manifold
\[
Z_\ge = \left\{ u = z+w(z,\ge) \mid z \in Z \right\}
\]
\begin{lemma}
$Z_\ge$ is a natural constraint for $f_\ge$: if $u =z+w(z,\ge)\in Z_\ge$ and $f'_{\ge \vert Z_\ge}(u)=0$, then $f_\ge'(u)=0$.
\end{lemma}
\begin{proof}
By assumption $f_\ge'(u)$ is orthogonal to $T_u z_\ge$. From the properties of $w$ it follows that $f_\ge'(u) \in T_z Z$, and $T_u Z_\ge$ is close to $T_z Z$ when $\ge$ is small. Therefore $f_\ge'(u)=0$.
\end{proof}
This Lemma allows us to replace the problem $f_\ge'(u)=0$ with the \emph{finite-dimensional} problem $f'_{\ge \vert Z_\ge}(u)=0$, $u \in Z_\ge$.

An expansion with respect to $\ge$,
\begin{align*}
f_\ge(z+w(z,\ge)) &= f_0(z+w(z,\ge))+\ge G(z + w(z,\ge)) \\
&= f_0(z) + \ge G(z) + o(\ge),
\end{align*}
shows the following existence result.\footnote{We recall that a critical point $u$ of a functional $f$ is called \emph{stable} whenever each functional $g$, sufficiently close to $f$ in the $C^1$-norm, has itself a critical point. See \cite{YYL}}
\begin{proposition} \label{prop:1.2}
Under our general assumptions ($h_1$), ($h_2$) and ($h_3$), the functional $f_\ge$ has at least one critical point, provided that $G_{\vert Z}$ has a stable critical point. In particular, this happens whenever there exist an open set $A \subset Z$ and a point $z_0 \in A$ such that
\[
G(z_0) < \inf_{\partial A} G \quad \text{(or $G(z_0) > \sup_{\partial A} G$})
\]
\end{proposition}
For a neat exposition of the complete theory we refer to the book \cite{AM}, where the interested reader will find many applications and generalizations.

\section{Pseudo-relativistic Hartree equations}

As a first application of the general principle, we study a class of equations involving a fractional differential operator:
\begin{equation} \label{eq:prh}
\sqrt{-\Delta + m^2} \ u + \mu u = \left( |x|^{-1} * |u|^2 \right) \left( 1+\ge g(x) \right)u \quad \text{in $\mathbb{R}^3$}.
\end{equation}
We suppose that $\mu>0$ is a given parameter.
This equation, when $\ge =0$, is related to models from astrophysics and boson stars (see \cite{CZNrend,CZN} for references). For the reader's sake, we recall that the operator
\[
\sqrt{-\Delta + m^2}
\]
can be defined on $f \in H^1(\mathbb{R}^3)$ by the following formula:
\[
\mathcal{F}\left( \left( \sqrt{-\Delta + m^2} \, f \right) \right) \left( \xi \right) = \sqrt{|\xi|^2+m^2} \, \mathcal{F}f (k),
\]
where $\mathcal{F}$ is the Fourier transform and $\xi \in \mathbb{R}^3$. An alternative approach to this fractional operator is through a \emph{local realization}: given any $u \in C_0^\infty(\mathbb{R}^3)$, there exists a unique solution $v \in C_0^\infty(\mathbb{R}_{+}^{3+1})$ of the Dirichlet problem
\begin{equation}
\begin{cases}
-\Delta v + m^2 y = 0 &\text{in $\mathbb{R}_{+}^{3+1}$} \\
v(0,y) = u(y) &\text{for $y \in \mathbb{R}^3 = \partial \mathbb{R}_{+}^{3+1}$}.
\end{cases}
\end{equation}
Here $\mathbb{R}_{+}^{3+1} = \left\{ (x,y) \mid x \in \mathbb{R}^3, \ y>0 \right\}$. Setting 
\[
Tu(y) = -\frac{\partial v}{\partial x} (0,y),
\]
the function $w(x,y) = -\frac{\partial v}{\partial x}(x,y)$ solves the problem
\begin{equation}
\begin{cases}
-\Delta w + m^2 w = 0 &\text{in $\mathbb{R}_{+}^{3+1}$} \\
w(0,y) = Tu(y) = -\frac{\partial v}{\partial x}(0,y) &\text{for $y \in \mathbb{R}^3 = \partial \mathbb{R}_{+}^{3+1}$},
\end{cases}
\end{equation}
and this implies that 
\[
T(Tu)(y) = -\frac{\partial w}{\partial x}(0,y) = \frac{\partial^2 w}{\partial x^2} (0,y) = \left(-\Delta_y v + m^2 v \right)(0,y)
\]
and therefore $T^2 = \left(-\Delta_y v + m^2 v \right)$. 

Since the group of translations acts on solutions to (\ref{eq:prh}), we face here a non-trivial lack of compactness: generally speaking, Palais-Smale sequences need not be (relatively) compact. Let us see how the perturbation technique presented in the Introduction may help us to overcome this issue.

To embed this problem into our abstract scheme, we set $H = H^{\frac12}(\mathbb{R})$,
the usual Sobolev space of fractional order that can also be seen as the trace space of $H^1(\mathbb{R}^{3+1}_{+})$, and
\begin{equation*}
f_\ge (u) = 
\frac{1}{2}\int_{\mathbb{R}^3} \left( \sqrt{-\Delta + m^2} \ |u|^2 + \mu |u|^2 \right)- \frac{1}{4} \int_{\mathbb{R}^3} \left( |x|^{-1} * |u|^2 \right) \left( 1 + \ge g(x) \right) |u|^2 \, dx.
\end{equation*}
We assume that $g \in  L^\infty (\mathbb{R}^3)$. Since the convolution kernel $x \mapsto |x|^{-1}$ belongs to the Lorentz space $L_w^3(\mathbb{R}^3)$, we can invoke the following \emph{weak Young inequality} to conclude that $f_\ge$ is well-defined:
\[
\int_{\mathbb{R}^3} \left( |x|^{-1} * |u|^2 \right)|u|^2 \leq C \||x|^{-1}\|_{L_w^3} \|u\|_{L^2} \|u\|_{L^{3}},
\]
where $C>0$ is a universal constant independent of $u \in H^{\frac12}(\mathbb{R}^3)$. We recall that the Lorentz space (or weak $L^3$ space)  $L_w^3(\mathbb{R}^3)$ is the set of those functions $f$ for which the quasi-norm\footnote{$\mathcal{L}$ denotes the Lebesgue measure in $\mathbb{R}^3$.}
\[
\|f \|_{L_w^3}^3 = \sup_{t>0} t^3 \mathcal{L}\left( \left\{ x \in \mathbb{R}^3 \mid |f(x)|>t \right\} \right)
\]
is finite. Moreover, we remark that $|\cdot |^{-1} \in L^r(\mathbb{R}^3) + L^\infty (\mathbb{R}^3)$, as is immediately seen by writing, for $R>0$,
\[
\frac{1}{|x|} = \frac{1}{|x|} \chi_{B(0,R)} + \frac{1}{|x|} \chi_{\mathbb{R}^3 \setminus B(0,R)}.
\]
It is also easy to check that $f_\ge \in C^2(H)$. For our setting we define
\begin{align*}
f_0 (u) &= \frac{1}{2}\int_{\mathbb{R}^3} \left( \sqrt{-\Delta + m^2} \ |u|^2 + \mu |u|^2 \right) - \frac{1}{4} \int_{\mathbb{R}^3} \left( |x|^{-1} * |u|^2 \right) |u|^2 \, dx \\
G(u) &= - \frac{1}{4} \int_{\mathbb{R}^3} \left( |x|^{-1} * |u|^2 \right) g(x) |u|^2 \, dx
\end{align*}
Let us recall some important facts proved in \cite{EL}. If
\[
\mathcal{E}(u) = \frac{1}{2}\int_{\mathbb{R}^3} \sqrt{-\Delta + m^2} \ |u|^2  - \frac{1}{4} \int_{\mathbb{R}^3} \left( |x|^{-1} * |u|^2 \right) |u|^2 \, dx,
\]
we consider the variational problem 
\begin{equation*}
E(N) =
 \inf \left\{ \mathcal{E}(u) \mid u\in H^{\frac12}(\mathbb{R}^3), \  \int_{\mathbb{R}^3} |u|^2 = N \right\}.
\end{equation*}
Then there exists a number (called the Chandrasekar mass) $N_c>4/\pi$ such that
\begin{itemize}
\item $E(N)$ is attained if and only if $0<N<N_*$; the corresponding minimizer $Q \in H^{\frac12}(\mathbb{R}^3)$ solves  
\[
\sqrt{-\Delta + m^2} \ Q + \mu Q = \left( |x|^{-1} * |Q|^2 \right) Q 
\]
for some Lagrange multiplier $\mu \in \mathbb{R}$.
\item Any minimizer $Q$ of $E(N)$ belongs to $H^s(\mathbb{R}^3)$ for all $s \geq 0$ and $Q$ decays exponentially fast at infinity.
\item Any minimizer $Q$ is radially decreasing and positive everywhere. 
\item If $N \ll 1$, then there exists one and only one minimizer $Q$ (up to translations), which is \emph{non-degenerate} in the following sense: the linearized operator
\[
L_{+} \xi = \left( \sqrt{-\Delta + m^2} + \mu \right) \xi - \left( |x|^{-1}*|Q|^2 \right)\xi - 2Q \left( |x|^{-1}*(Q\xi) \right)
\]
satisfies the condition
\[
\ker L_{+} = \operatorname{span} \left\{ \frac{\partial Q}{\partial x_1}, \frac{\partial Q}{\partial x_2}, \frac{\partial Q}{\partial x_3} \right\}.
\]
\end{itemize}
\begin{remark}
The Lagrange multiplier $\mu$ cannot be discarded. From a technical viewpoint this is due to the lack of scaling properties for $\sqrt{-\Delta + m^2}$. But this also reflects the fact that $E(N)$ is attained only if $N$ is strictly smaller than the Chandrasekar mass.
\end{remark}
From this moment we fix $N \ll 1$ and its Lagrange multiplier $\mu$ in such a way that $E(N)$ is uniquely solvable by a non-degenerate element $Q$. 
Without loss of generality, we can assume $m=1$.

We define the manifold
\[
Z = \left\{ Q(\cdot - \xi) \mid \xi \in \mathbb{R}^3 \right\}.
\] 
Since the Euler-Lagrange equation associated to critical points of $f_0$ is invariant under translations, each element of $Z$ is a critical point of $f_0$.
\begin{lemma}
The linear operator $f_0''(z)$ is Fredholm of index zero at every $z \in Z$.
\end{lemma}
\begin{proof}
Since
\[
f_0(u) = \frac{1}{2} \|u\|^2_{H} - \frac{1}{4} \int_{\mathbb{R}^3} \left( |x|^{-1}*|u|^2 \right)|u|^2\, dx,
\]
it suffices to check that $f_0''(z)$ is a compact perturbation of the identity. Therefore, we need to check that
\[
K(v_n,w_n) = \int_{\mathbb{R}^3} \left( |x|^{-1}*|Q|^2 \right) v_n w_n + 2 Q \left( |x|^{-1}*(Q v_n) \right)w_n \to 0
\]
whenever $\{v_n\}$ and $\{w_n\}$ are bounded sequences in $H^{\frac{1}{2}}(\mathbb{R}^3)$. Of course, we can assume without loss of generality that $v_n \rightharpoonup 0$ and $w_n \rightharpoonup 0$.

The first term in $K(v_n,w_n)$ goes to zero because it can be seen as a multiplication operator with $|x|^{-1}*|Q|^2$, and
\[
\lim_{|x| \to +\infty} |x|^{-1}*|Q|^2 =0.
\]
This was proved in \cite[Lemma 2.13]{CSS} in a slightly different context.  We recall the short argument in \cite{CZNrend}: given $\epsilon>0$, fix $\rho>0$ such that 
\[
\sup \left\{ \frac{1}{|y|} \mid |y| > \rho \right\} < \frac{\epsilon}{2}.
\]
Then
\begin{multline} \label{eq:sup}
\int_{\mathbb{R}^3} \frac{Q(y)^2}{|y-\zeta|} \, dy = \int_{B(\zeta,\rho)} \frac{Q(y)^2}{|y-\zeta|} \, dy 
+ \int_{\mathbb{R}^3 \setminus B(\zeta,\rho)} \frac{Q(y)^2}{|y-\zeta|} \, dy \\
\leq \||x|^{-1}\|_{L^r} \left( \int_{B(\zeta,\rho)} Q(y)^{2r'}\, dy \right)^{1/r'} + \frac{\epsilon}{2}\|Q\|_{L^2}^2
\end{multline}
for $r>3/2$, and we conclude by letting $|\zeta| \to +\infty$.

As for the second term, fix $\delta>0$ and $R>0$. Define
\[
\K_\delta (x) = \begin{cases}
1/|x| &\text{if $|x| \leq 1/\delta$}\\
0 &\text{otherwise}
\end{cases}
\]
and
\[
\K_\delta^R(x) = \max \left\{ |\K_\delta(x) - R ,0\right\} \chi_{B(0,R)}(x) + \K_\delta (x) \chi_{\mathbb{R}^3 \setminus B(0,R)}(x).
\]
It is clear that, given $\delta >0$, $\lim_{R \to +\infty} \|\K_\delta^R\|_{L^r} =0$ for any $r \in [1,3)$. Finally, set $\Theta_\delta^R = \K_\delta - \K_\delta^R$, and remark that $\operatorname{supp}\Theta_\delta^R \subset B(0,R)$.

We apply Young's inequality and get
\begin{multline*}
\int_{\mathbb{R}^3} \left( |x|^{-1}*|Q v_n| \right) |Q w_n| \\
\leq \int_{\mathbb{R}^3} \left( \left( |x|^{-1}-\K_\delta \right)*|Qv_n| \right)|Qw_n| + \int_{\mathbb{R}^3} \left( \K_\delta^R * |Qv_n| \right) |Qw_n| \\
{}+ \int_{\mathbb{R}^3} \left( \Theta_\delta^R * |Qv_n| \right) |Qw_n| \\
\leq \delta \|Q v_n\|_{L^1} \|Q w_n\|_{L^1} + \|\K_\delta^R\|_{L^{3/2}} \| Q v_n \|_{L^{3/2}} \|Qw_n\|_{L^{3/2}} + \int_{\mathbb{R}^3} \left( \Theta_\delta^R * |Qv_n| \right) |Qw_n| \\
\leq C \left( \delta + \|\K_\delta^R\|_{L^{3/2}} \right) + \int_{\mathbb{R}^3} \left( \Theta_\delta^R * |Qv_n| \right) |Qw_n|.
\end{multline*}
We need to prove that
\begin{equation} \label{eq:6}
\lim_{n \to +\infty} \int_{\mathbb{R}^3} \left( \Theta_\delta^R * |Qv_n| \right) |Qw_n| =0.
\end{equation}
To this aim, pick any $\epsilon>0$ and choose a radius $R_1>0$ such that 
\[
\int_{\mathbb{R}^3 \setminus B(0,R_1)} |Q|^2 < \epsilon.
\]
Then write $R_2 = R_1 + R$, so that $\Theta_\delta^R (z-y)=0$ whenever $|y| < R_1$ and $|z| \geq R_2$. Now,
\begin{multline*}
\int_{\mathbb{R}^3} \left( \Theta_\delta^R * |Qv_n| \right) |Qw_n| \\
= \int_{B(0,R_1)} \left( \Theta_\delta^R * \left( \chi_{B(0,R_2)} |Qv_n| \right) \right) |Qw_n| + \int_{\mathbb{R}^3 \setminus B(0,R_1)} \left( \Theta_\delta^R *  |Qv_n| \right)  |Qw_n| \\
\leq R \left( \int_{B(0,R_2)} |Qv_n| \right) \left( \int_{B(0,R_1)} |Qw_n| \right) \\
{}+ \|\Theta_\delta^R * |Qv_n| \|_\infty \left( \int_{\mathbb{R}^3 \setminus B(0,R_1)} |w_n|^2 \right)^{\frac{1}{2}} \left( \int_{\mathbb{R}^3 \setminus B(0,R_1)} |Q|^2 \right)^{\frac{1}{2}} \\
\leq R \|Q\|_{L^2} \|w_n\|_{L^2} \left( \left( \int_{B(0,R_2)} |v_n|^2 \right)^{\frac{1}{2}} \|Q\|_{L^2}+ R \|v_n\|_{L^2} \left( \int_{\mathbb{R}^3 \setminus B(0,R_1)} |Q|^2 \right)^{\frac{1}{2}} \right) \\
\leq C R \left(  \left( \int_{B(0,R_2)} |v_n|^2 \right)^{\frac{1}{2}} + \left( \int_{\mathbb{R}^3 \setminus B(0,R_1)} |Q|^2 \right)^{\frac{1}{2}} \right),
\end{multline*}
and the right-hand side goes to zero because $v_n \to 0$ strongly in $L_{\mathrm{loc}}^2$ thanks to Lemma \ref{lem:3.1}. Our claim (\ref{eq:6}) follows by letting $n \to +\infty$, $R \to +\infty$ and finally $\delta \to 0$. 
\end{proof}
As a consequence, our problem fits into the abstract framework. Here is a possible existence result.
\begin{theorem} \label{th:2.3}
Pick $N>0$ so small that the variational problem $E(N)$ has a unique non-degenerate ground state $Q \in H^{\frac12}(\mathbb{R}^3)$; let $\mu \in \mathbb{R}^3$ be the corresponding Lagrange multiplier. Assume moreover that $g \in L^\infty (\mathbb{R}^3)$ vanishes at infinity and does not change sign. Then, for every $\ge$ sufficiently small, equation (\ref{eq:prh}) has (at least) a solution $u_\ge \in H^{\frac12}(\mathbb{R}^3)$ such that $u_\ge \simeq Q(\cdot - \xi)$, for a suitable choice of $\xi \in \mathbb{R}^3$.
\end{theorem}
\begin{proof}
Set $\widetilde{Q} = \left( |x|^{-1}*|Q|^2 \right)|Q|^2 \in L^1(\mathbb{R}^3)$, and recall that $|x|^{-1}*|Q|^2 \in L^\infty (\mathbb{R}^3)$ by (\ref{eq:sup}). The abstract scheme invites us to looking for stable critical points of the (finite--dimensional) function
\[
G(\xi) = \int_{\mathbb{R}^3} g(x) \widetilde{Q}(x - \xi) \, dx, \qquad \xi \in \mathbb{R}^3.
\]
Let us decompose $G(\xi) = G_1(R,\xi)+G_\infty (R,\xi)$, where $R>0$ and
\begin{align*}
G_1 (\xi) &= \int_{|x|<R} g(x) \widetilde{Q}(x - \xi) \, dx \\
G_\infty (\xi) &= \int_{|x| \geq R} g(x) \widetilde{Q}(x - \xi) \, dx.
\end{align*}
We can estimate, for some constant $C_\infty >0$,
\begin{equation} \label{eq:G-inf}
\left| G_\infty (R,\xi) \right| \leq \sup_{|x| \geq R} |g(x)| \int_{|x-\xi| \geq R} \widetilde{Q}(x) \, dx \leq C_\infty \sup_{|x| \geq R} |g(x)|.
\end{equation}
On the other hand,
\begin{equation} \label{eq:G-1}
\left| G_1(R,\xi) \right| \leq \sup_{|x|<R} |g(x)| \int_{|x-\xi|<R} \widetilde{Q}(x)\, dx.
\end{equation}
Let $\epsilon >0$. Fix $R = R(\epsilon)>0$ such that $\sup_{|x| \geq R} |g(x)| < \epsilon$. The right-hand side of (\ref{eq:G-1}) tend to zero as $|\xi| \to +\infty$, and thus
\[
\limsup_{|x| \to +\infty} |G(\xi)| \leq C_\infty \epsilon.
\]
Since this is true for any $\epsilon>0$, we conclude that
\[
\lim_{|\xi| \to +\infty} G(\xi) = 0,
\]
and that $G$ does not change sign. Therefore $G$ must have a strict local maximum (or minimum) point at some $\xi_0$. It now suffices to apply Proposition \ref{prop:1.2}.
\end{proof}
By exploiting the exponential decay of $Q$, we can prove a similar result under different assumptions.
\begin{theorem} \label{th:2.4}
Pick $N>0$ so small that the variational problem $E(N)$ has a unique non-degenerate ground state $Q \in H^{\frac12}(\mathbb{R}^3)$; let $\mu \in \mathbb{R}^3$ be the corresponding Lagrange multiplier. Assume moreover that $g \in L^\theta (\mathbb{R}^3)$ for some $\theta>1$ and $g$ does not change sign. Then, for every $\ge$ sufficiently small, equation (\ref{eq:prh}) has (at least) a solution $u_\ge \in H^{\frac12}(\mathbb{R}^3)$ such that $u_\ge \simeq Q(\cdot - \xi)$, for a suitable choice of $\xi \in \mathbb{R}^3$.
\end{theorem}
\begin{proof}
Once more, $G$ does not change sign, and $G(\xi) = G_1(R,\xi)+G_\infty (R,\xi)$. Now,
\begin{align*}
\left| G_\infty (R,\xi)\right| &\leq \left( \int_{|x| \geq R} |g(x)|^\theta \, dx \right)^{\frac{1}{\theta}} \left( \int_{|x-\xi|\geq R} \left| \widetilde{Q}(x) \right|^{{\theta^\prime}} \, dx \right)^{\frac{1}{\theta^\prime}}\\
&\leq \|\widetilde{Q}\|_{L^{\theta^\prime}} \left( \int_{|x| \geq R} |g(x)|^\theta \, dx \right)^{\frac{1}{\theta}}.
\end{align*}
Similarly,
\begin{align*}
\left| G_1(R,\xi) \right| &\leq \left( \int_{|x|<R} |g(x)|^\theta \, dx \right)^{\frac{1}{\theta}} \left( \int_{|x-\xi| < R} \left| \widetilde{Q}(x) \right|^{\theta^\prime} \right)^{\frac{1}{\theta^\prime}} \\
&\leq \|g\|_{L^\theta} \left( \int_{|x-\xi| < R} \left| \widetilde{Q}(x) \right|^{\theta^\prime} \right)^{\frac{1}{\theta^\prime}}.
\end{align*}
In these estimates, $1/\theta + 1 / \theta^\prime =1$, and $\widetilde{Q} \in L^{\theta^\prime}(\mathbb{R}^3)$ thanks to the exponential decay. Letting $R \to +\infty$ and then $|\xi| \to +\infty$, we conclude as before that $\lim_{|\xi| \to +\infty} G(\xi)=0$, and consequently $G$ attains either a strict local maximum or a strict local minimum (or both).
\end{proof}
\begin{remark}
The assumption that $g$ has constant sign ensures that $G$ is non-constant. We can also allow sign-changing perturbations $g$, for example under the assumption that $\int_{\mathbb{R}^3} g(x) \widetilde{Q}(x) \, dx = G(0) \neq 0$.
\end{remark}
\begin{remark}
The reader will realize that we can deal with perturbed equations other than (\ref{eq:prh}). For example we could prove a similar existence result for
\[
\sqrt{-\Delta + m^2} \, u + \mu u = \left( |x|^{-1}*|u|^2 \right) u + \ge g(x) |u|^{p-1}u,
\]
provided that the integral $\int_{\mathbb{R}^3} g(x) |u(x)|^{p+1} \, dx$
is finite for every $u \in H$.
\end{remark}

%%%
\section{A model with the fractional laplacian}
    
    Equations governed by fractional powers of the Laplace operator $\Delta$
arise in several physical models, and we refer to \cite{FL} for some discussion. In this section we show that some existence results can be easily proved also for some of these problems.
    
   Let us consider the problem
\begin{equation}\label{eq:1}
   (-\Delta)^s u + u = \left(1+\ge h(x) \right) |u|^{p}u \quad \text{in $\mathbb{R}$},
   \end{equation}
where $0<s<1$ and
\[
0<p<p^\dag = \begin{cases} \frac{4s}{1-2s} &\text{in $0<s<1/2$} \\ +\infty &\text{if $1/2 \leq s <1$}. \end{cases}
\]
The function $h$ is a \emph{bounded} potential and $\ge>0$ is a ``small''
parameter. Our aim is to find 
solution of (\ref{eq:1}) as $\ge \to  0$.

Equations of this form have been widely investigated in the last decades when
$s=1$, i.e. when the differential 
operator coincides with the standard laplacian. For fractional operators, to the
best of our knowledge, the literature is still growing. With $\ge=1$ and in
dimension $N \geq 2$, a similar problem is studied in \cite{FQT} and in
\cite{DPV} under suitable assumptions on $h$. See also references therein.

\medskip

We wish to spend a few words on the operator $(-\Delta)^s$ appearing on the
left-hand side of (\ref{eq:1}). 
This operator is called \emph{fractional laplacian} (of order $s$) and there are
several \emph{almost equivalent} definitions. A first approach is to regard this
operator via Fourier analysis: for every test function $\varphi$,
\[
(-\Delta)^s\varphi(\xi)= \mathcal{F}^{-1} \left( |\xi|^{2s} \mathcal{F}(\varphi)(\xi) \right),
\]
where $\mathcal{F}$ stands for the Fourier transform. Hence $(-\Delta^s)$ is
pseudo-differential 
operator with symbol $|\xi|^{2s}$.

Equivalently, we may define\footnote{P.V. denotes here the \emph{principal value} of the integral.}
\[
(-\Delta^s) \varphi(x) = C_{s} \operatorname{P.V.} \int_{\mathbb{R}} \frac{\varphi(x)-\varphi(y)}{|x-y|^{1+2s}}dy,
\]
where
\[
C_{s}=\left( \int_\mathbb{R} \frac{1-\cos x}{|x|^{1+2s}} dx
\right)^{-1}.
\]
Either definition makes it clear that $(-\Delta^s)$ is a \emph{non-local}
operator: unlike the standard 
laplacian ($s=1$), the value of $(-\Delta^s)\varphi$ depends on the values of
$\varphi$ in the whole $\mathbb{R}$. In particular, compactly supported
functions won't have, in general, compactly supported fractional laplacians.
This is a serious obstruction to the use of standard techniques of nonlinear
differential equations such as localization and cut-offs.

Recently, Caffarelli and Silvestre proved in \cite{CS} a very interesting
\emph{local realization} 
of the fractional laplacian via a Dirichlet-to-Neumann operator. Roughly
speaking, we can add one more variable and solve the linear problem in
$\mathbb{R}^2_{+} = \left\{ (x,y) \mid x \in \mathbb{R}, \ y>0 \right\}$
\[
\begin{cases}
-\operatorname{div} \left(y^{1-2s} \nabla u \right) = 0 &\text{in $\mathbb{R}^2_{+}$}\\
u(x,0)=\varphi(x) &\text{for $x \in \mathbb{R}$}.
\end{cases}
\]
Then
\[
(-\Delta^s \varphi)(x)=-b_s \lim_{y \to 0+} y^{1-2s} \frac{\partial u}{\partial y}
\]
for a suitable constant $b_s$. This approach is very useful for solving
differential equations, 
as we can work again in a local setting. Once more, it is clear from either definition that problems like (\ref{eq:1}) are non-compact, due to the action of the group of translations.
 
 In the sequel, the r\^{o}le of $(-\Delta)^s$ will be somehow hidden in the
\emph{unperturbed problem} (see below for the definition), and we will use
the definition via Fourier analysis only for definiteness.

A function space that can be used quite naturally to study equation (\ref{eq:1})
is the fractional
Sobolev space\footnote{We have set $\hat{u}=\mathcal{F}(u)$, as usual.}
\[
H^s(\mathbb{R}) = \left\{ u \in L^2(\mathbb{R}) \mid \int_{\mathbb{R}} |\xi|^{2s}|\hat{u}|^2 \, d\xi < \infty \right\}
\]
endowed with the norm
\[
\|u\|_{H^s}^2 = \|u\|_{L^2}^2 + \int_{\mathbb{R}} |\xi|^{2s} |\hat{u}|^2 \, d\xi.
\]
We recall an embedding property for this space. For more information about fractional Sobolev spaces, we recommend the recent survey \cite{DNPV}.

\begin{lemma} \label{lem:3.1}
Let $2 \leq q \leq 2_s^\star = 2/(1-2s)$. Then $H^s(\mathbb{R})$ is continuously
embedded into 
$L^q(\mathbb{R})$. Moreover, this embedding is locally compact provided that $2
\leq q < 2_s^\star$.
\end{lemma}
\begin{remark}
Even in dimension one, functions in $H^s(\mathbb{R})$ need not be continuous. Actually, Sobolev's critical exponent for $H^s$ is $2/(1-2s)$: continuity is granted only when $1/2<s<1$.
\end{remark}

We assume that $h \in C(\mathbb{R})$ is bounded. Then we can easily prove that
(weak) solutions
to (\ref{eq:1}) correspond to critical points of the functional
\[
f_\ge (u) = \frac{1}{2} \|u\|_{H^s}^2 - \frac{1}{p+2} \int_{\mathbb{R}} |u|^{p+2} - \frac{\ge}{p+2} \int_{\mathbb{R}} h(x) |u(x)|^{p+2}\, dx.
\]
Setting
\[
f_0(u)=\frac{1}{2} \|u\|_{H^s}^2 - \frac{1}{p+2} \int_{\mathbb{R}} |u|^{p+2},
\]
we can write
\begin{equation}\label{eq:2}
f_{\ge}(u)=f_0(u)+\ge G(u),
\end{equation}
where
\[
G(u)=-\frac{1}{p+2}\int_{\mathbb{R}} h(x) |u(x)|^{p+2}\, dx.
\]
When $\ge=0$, critical points of the functional $f_0 \colon H^s(\mathbb{R}) \to
\mathbb{R}$ 
correspond to (weak) solutions to the equation
\begin{equation}\label{eq:3}
(-\Delta^s)u+u=|u|^{p}u \quad\text{in $\mathbb{R}$,}
\end{equation}
which we call \emph{the unperturbed problem}. This equation was deeply studied in
\cite{FL}. 
We recall here the main results.
\begin{theorem}[Franck and Lenzmann] \label{th:2.1}
Let $0<s<1$ and $0<p<p^\dag$. Then the following holds.
\begin{itemize}
\item[(i)] \textbf{Existence:} There exists a solution $Q \in H^s(\mathbb{R})$
of equation (\ref{eq:3}) such that $Q=Q(|x|)>0$ is even, positive and strictly
decreasing in $|x|$. Moreover, the function $Q \in H^s(\mathbb{R})$ is a
minimizer for 
\[
J^{s,p}(u)= \frac{\left(\int (-\Delta)^{\frac{s}{2}}u|^2\right)^{\frac{p}{4s}} \left( \int |u|^2 \right)^{\frac{p}{4s} (2s-1)+1}}{\int |u|^{p+2}}.
\]
\item[(ii)] \textbf{Symmetry and Monotonicity:} If $Q \in H^s(\mathbb{R})$ with $Q \geq 0$ and $Q \not\equiv 0$ solves (\ref{eq:3}), then there exists $x_0 \in \mathbb{R}$ such that $Q(\cdot - x_0)$ is even, positive and strictly decreasing in $|x-x_0|$.
\item[(iii)] \textbf{Regularity and Decay:} If $Q \in H^s(\mathbb{R})$ solves (\ref{eq:3}), then $Q \in H^{2s+1}(\mathbb{R})$. Moreover, we have the decay estimate
\[
|Q(x)|+|xQ'(x)| \leq \frac{C}{1+|x|^{2s+1}}
\]
for all $x \in \mathbb{R}$ and some constant $C>0$.
\end{itemize}
\end{theorem}
\begin{remark}
 Unlike the familiar case $s=1$, ground state solutions $Q$ do \emph{not} decay
exponentially fast at infinity.
\end{remark}

\begin{theorem}[Franck and Lenzmann] \label{th:2.2}
Let $0<s<1$ and $0<p<p^\dag$. Suppose that $Q \in H^s(\mathbb{R})$ is a positive
solution of (\ref{eq:3}) 
and consider the linearized operator
\[
L_{+}=(-\Delta)^s +I -(p+1)Q^p
\]
acting on $L^2(\mathbb{R})$. Then the following condition holds: If $Q \in
H^s(\mathbb{R})$ is a 
local minimizer for $J^{s,p}$, then $L{+}$ is non degenerate, i.e. its kernel
satisfies
\[
\ker L{+} = \operatorname{span}\{Q'\}.
\]
In particular, any ground state solution $Q=Q(|x|)$ of equation (\ref{eq:3}) has
a non degenerate 
linearized operator $L_{+}$.
\end{theorem}
\begin{theorem}[Franck and Lenzmann]
Let $0<s<1$ and $0<p<p^\dag$. Then the ground state solution $Q=Q(|x|)>0$ for
equation 
(\ref{eq:3}) is unique.
\end{theorem}

We now introduce our manifold
\[
Z=\left\{ Q( \cdot - \theta) \mid \theta \in \mathbb{R} \right\},
\]
where $Q$ is the unique, radially symmetric, positive ground state solution of
(\ref{eq:3}). Each element of $Z$ is a critical point of $f_0$; moreover, since $Q$ decays at infinity, 
it is standard to check that $D^2 f_0(z)$ is a compact perturbation of the
identity, for every $z \in Z$, and assumption (h$_2$) is thus matched.

By the same token as in the previous section, we can prove the next existence result.

\begin{theorem}
Assume $h \in L^\infty(\mathbb{R})$ has constant sign and $\lim_{|x| \to +\infty} h(x)=0$. Suppose that $0<s<1$ and $0<p<p^\dag$. Then, for
every $\ge$ 
sufficiently small, equation (\ref{eq:1}) has a non-trivial solution $u_\ge
\simeq Q(\cdot - \bar{\theta})$, for some $\bar{\theta} \in \mathbb{R}$.
\end{theorem}
\begin{proof}
We need to find stable critical points of the function
\[
G(\theta) = \int_{\mathbb{R}} h(x) |Q(x-\theta)|^{p+2} \, dx.
\]
As in the proof of Theorem \ref{th:2.3}, $G(\theta) \to 0$ as $\theta \to \pm \infty$, and $\operatorname{sign} G = \operatorname{sign} h$. Hence $G$ has either a strict minimum or a strict maximum point (or both), and we conclude.
\end{proof}
Also in this case we can modify the assumptions on $h$ and replace them by some integrability condition. However, the solution $Q$ no longer decays exponentially fast at infinity, and we must be more precise.
\begin{theorem}
Suppose that $0<s<1$, $0<p<p^\dag$,
and let $h \in L^\theta(\mathbb{R})$ with
\[
\theta =
\begin{cases}
\frac{2}{p+2s(p-2)} &\text{if $0<s<\frac{1}{2}$} \\
\text{any number} &\text{if $\frac{1}{2} \leq s <1$}.
\end{cases}
\] 
If $h$ does not change sign, then, for
every $\ge$ 
sufficiently small, equation (\ref{eq:1}) has a non-trivial solution $u_\ge
\simeq Q(\cdot - \bar{\theta})$, for some $\bar{\theta} \in \mathbb{R}$.
\end{theorem}
\begin{proof}
The proof is similar to that of Theorem \ref{th:2.4}. However, when applying H\"{o}lder's inequality, we must be sure that
\[
\int_{\mathbb{R}^3} \left| Q(x) \right|^{(p+2)\theta^\prime}\, dx < \infty,
\]
where $1/\theta+1/\theta^\prime =1$.
This is true if
\[
\theta^\prime (p+2) = \frac{2}{1-2s}, \quad \text{if $0<s<\frac{1}{2}$},
\]
which boils down to 
\[
\theta = \frac{2}{p+2s(p-2)} \quad\text{if $0<s<\frac{1}{2}$}.
\]
The case $1/2 \leq s <1$ is easier.
\end{proof}

\begin{remark}
Of course stable critical points of $G$ may also occur under different assumptions on $h$. Moreover, different stable critical points of $G$ give rise to different solutions of (\ref{eq:1}); if we know that $G$ has two (ore more) stable critical points, then our equation will have two (or more) solutions.
\end{remark}

\section{Comments and perspectives}

\begin{itemize}
\item It is easy to use the regularity estimates of \cite{CZNrend,CZN} and \cite{FQT} to prove that our solutions have additional regularity.
\item Although obtained by rather easy considerations, we believe that our
existence results are new and cannot be easily recovered by adapting the corresponding
methods for the unperturbed problems. We wish to remark that, to the best of our knowledge, for fractional operators there is no precise analysis of loss of compactness in the existing literature.
In the recent preprint \cite{Mu} the author's assumptions allow non-constant potential functions, but they should be radially symmetric. We do not expect our solutions to be necessarily invariant under rotations.
\item Unlike the operator $\sqrt{-\Delta +m^2}$, the fractional laplacian $(-\Delta)^s$ scales in a standard way: under a dilation $x \mapsto \epsilon x$, the fractional laplacian becomes $\epsilon^{2s}(-\Delta)^s$. Therefore it is tempting to investigate the \emph{singularly perturbed} equation
\begin{equation} \label{eq:9}
\epsilon^{2s}(-\Delta)^s + V(x)u = |u|^p u,
\end{equation}
where $V \colon \mathbb{R} \to \mathbb{R}$ is ax external potential function. When $s=1$, i.e. when the fractional laplacian reduces to the local Laplace operator, equations like (\ref{eq:9}) appear in a lot of papers. Roughly speaking, \emph{single-peak} solutions, i.e. solutions that concentrate at some point as $\epsilon \to 0$, are generated by ``good'' critical points of $V$. These solutions can also be discovered by a suitable modification of the perturbation developed in \cite{ABC} and generalized in \cite{AMS}. See also \cite{AM} for a survey. However, it seems that the slow decay at infinity of solutions to $(-\Delta)^s u + u = |u|^{p}u$ is a severe obstruction. We believe that the analysis of singularly perturbed problems for non-local operators should be pursued further.
\item Another interesting issue is to extend the results of non-degeneracy 
for the fractional laplacian to higher dimensions. The main ingredient for our
approach to work again is a
version of Theorem \ref{th:2.2} in general dimension $N>1$. The proof of
\cite{FL} is heavily based on results about the number of zeroes of an
eigenfunction corresponding to the second eigenvalue for operators like
$(-\Delta)^s+V$, where $V$ is a suitable potential (see the comments in
\cite{F}). This approach cannot be immediately generalized to any space
dimension.
\end{itemize}

    \nocite{*}
%\printbibliography
\bibliographystyle{plain}
\bibliography{pert}

\begin{thebibliography}{10}

\bibitem{ABC}
Antonio Ambrosetti, Marino Badiale, and Silvia Cingolani.
\newblock Semiclassical states of nonlinear {S}chr{\"o}dinger equations.
\newblock {\em Arch. Rational Mech. Anal.}, 140(3):285--300, 1997.

\bibitem{AM}
Antonio Ambrosetti and Andrea Malchiodi.
\newblock {\em Perturbation methods and semilinear elliptic problems on {${\bf
  R}^n$}}, volume 240 of {\em Progress in Mathematics}.
\newblock Birkh\"auser Verlag, Basel, 2006.

\bibitem{AMS}
Antonio Ambrosetti, Andrea Malchiodi, and Simone Secchi.
\newblock Multiplicity results for some nonlinear {S}chr{\"o}dinger equations
  with potentials.
\newblock {\em Arch. Ration. Mech. Anal.}, 159(3):253--271, 2001.

\bibitem{BB}
Massimiliano Berti and Philippe Bolle.
\newblock Homoclinics and chaotic behavior for perturbed second order systems.
\newblock {\em Annali di Matematica Pura ed applicata}, CLXXVI:323--378, 1999.

\bibitem{CS}
Luis Caffarelli and Luis Silvestre.
\newblock An extension problem related to the fractional {L}aplacian.
\newblock {\em Comm. Partial Differential Equations}, 32(7-9):1245--1260, 2007.

\bibitem{CSS}
Silvia Cingolani, Simone Secchi, and Marco Squassina.
\newblock Semi-classical limit for {S}chr\"odinger equations with magnetic
  field and {H}artree-type nonlinearities.
\newblock {\em Proc. Roy. Soc. Edinburgh Sect. A}, 140(5):973--1009, 2010.

\bibitem{CZNrend}
Vittorio Coti~Zelati and Margherita Nolasco.
\newblock Existence of ground states for nonlinear, pseudorelativistic
  {S}chr\"{o}dinger equations.
\newblock {\em Rend. Lincei Mat. Appl.}, 22:51--72, 2011.

\bibitem{CZN}
Vittorio Coti~Zelati and Margherita Nolasco.
\newblock Ground states for pseudo-relativistic {H}artree equations of critical
  type.
\newblock Preprint 2012.

\bibitem{DNPV}
Eleonora Di~Nezza, Giampiero Palatuci, and Enrico Valdinoci.
\newblock Hitchhiker's guide to the fractional sobolev spaces.
\newblock {\em Bull. Sci. Math.}, 2012.

\bibitem{DPV}
Serena Dipierro, Giampiero Palatucci, and Enrico Valdinoci.
\newblock Existence and symmetry results for a schr\"{o}dinger type problem
  involving the fractional laplacian.
\newblock Preprint 2012.

\bibitem{FQT}
Patricio Felmer, Alexander Quaas, and Jinggan Tan.
\newblock Positive solutions of the nonlinear schr\"{o}dinger equation with the
  fractional laplacian.
\newblock To appear on \emph{Proc. Roy. Soc. Edinburgh} Sect A.

\bibitem{F}
Rupert~L Frank.
\newblock On the uniqueness of ground states of non-local equations.
\newblock Preprint, 2011.

\bibitem{FL}
Rupert~L. Frank and Enno Lenzmann.
\newblock Uniqueness and nondegeneracy of ground states for $(-\delta)^s q+ q -
  q^{\alpha +1}=0$ in $\mathbb{R}$.
\newblock To appear on \emph{Annals of Mathematics}.

\bibitem{EL}
Enno Lenzmann.
\newblock Uniqueness of ground states for pseudorelativistic hartree equations.
\newblock {\em Analysis and PDE}, 2(1), 2009.

\bibitem{YYL}
YanYan Li.
\newblock On a singularly perturbed elliptic equation.
\newblock {\em Adv. Differential Equations}, 2(6):955--980, 1997.

\bibitem{Mu}
Dimitri Mugnai.
\newblock The pseudorelativistic hartree equation with a general nonlinearity:
  existence, non existence and variational identities.
\newblock Preprint 2012.

\end{thebibliography}

\par\bigskip\noindent {\scshape Dipartimento di Matematica ed Applicazioni, Universit\`a di Milano-Bicocca, via Cozzi 53, I-20125 Milano. Email:} \href{mailto:simone.secchi@unimib.it}{simone.secchi@unimib.it}
    \end{document}